\documentclass[11pt]{article}

\usepackage{graphicx}
\usepackage{amsthm,amsmath,amssymb,verbatim,fancyhdr}
\usepackage{array,url}
\usepackage{color}

\newtheorem{thm}{Theorem}
\newtheorem{lem}[thm]{Lemma}
\newtheorem{cor}[thm]{Corollary}

\newtheorem{problem}{Problem}

\def\VEC#1#2#3{#1_{#2},\ldots,#1_{#3}}

\def\qed{\ifhmode\unskip\nobreak\hfill$\Box$\bigskip\fi \ifmmode\eqno{Box}\fi}
\newcommand{\RR}{\mathbb{R}}

\begin{document}

\title{Fractional chromatic number of a random subgraph}

\author{
  Bojan Mohar\thanks{Supported in part by the NSERC Discovery Grant R611450 (Canada),
  by the Canada Research Chairs program,
  and by the Research Project J1-8130 of ARRS (Slovenia).}\\[1mm]
  Department of Mathematics\\
  Simon Fraser University\\
  Burnaby, BC, Canada\\
  {\tt mohar@sfu.ca}
\and
  Hehui Wu\thanks{Part of this work was done while the author was a PIMS Postdoctoral Fellow at the Department of Mathematics, Simon Fraser University, Burnaby, B.C.}\\[1mm]
  Shanghai Center for Mathematical Sciences\\
  Fudan University\\
  Shanghai, China\\
  {\tt hhwu@fudan.edu.cn}
}

\maketitle

\begin{abstract}
It is well known \cite{Erd} that a random subgraph of the complete graph $K_n$ has chromatic number $\Theta(n/\log n)$ w.h.p. Boris Bukh asked whether the same holds for a random subgraph of any $n$-chromatic graph, at least in expectation. In this paper it is shown that for every graph, whose fractional chromatic number is at least $n$, the fractional chromatic number of its random subgraph is at least $n/(8\log_2(4n))$ with probability more than $1-\frac{1}{2n}$. This gives the affirmative answer for a strengthening of Bukh's question for the fractional chromatic number.
\end{abstract}

\section{Introduction}

If $G$ is a graph and $p \in (0,1)$, we let $G_p$ denote a \emph{random subgraph} of $G$ where each edge of $G$ appears in $G_p$ independently with probability $p$.
The standard Erd\H{o}s-Renyi random graph $G_{n,p}$ can be viewed as a random subgraph of the complete graph $K_n$.  A lot is known about properties of random graphs, but not so much about the generalized notion of random subgraphs. In particular, it is known that the chromatic number of $G_{n,p}$ is $\Theta(n/\log_{1/p}n)$, almost surely, for a wide range of values $p=p(n)$. Boris Bukh \cite{Bukh} asked whether the same phenomenon occurs for random subgraphs of any $n$-chromatic graph, at least in expectation, when $p$ is constant.

\begin{problem}[Bukh]
\label{prb:Bukh}
Does there exist a constant $c$ so that for every graph $G$, the expected chromatic number of its random subgraphs satisfies:
$${\mathbb E}(\chi(G_{1/2})) > c \, \frac{\chi(G)}{\log \chi(G)}\,?$$
\end{problem}

Only very special cases have been considered in the past.
Bogolyubskiy et al. \cite{BGPR1,BGPR2} considered random subgraphs of certain distance graphs and their chromatic number. Kupavskii \cite{Ku16} studied the chromatic number of a random subgraph of Kneser and Schrijver graphs $KG(n,k)$ and $SG(n,k)$, as $n$ grows. For a wide range of parameters $k=k(n)$ and $p=p(n)$, he proved that $\chi(KG_p(n,k))$ is very close to $\chi(KG_p(n,k))$ w.h.p., differing by at most 4 in many cases. His work was preceded by pioneering work of Bollob\'as, Narayanan and Raigorodskii \cite{BNR} and Balogh, Bollob\'as, and Narayanan \cite{BBN}, who studied independent sets in random subgraphs of Kneser graphs $KG(n,k)$.

While Problem \ref{prb:Bukh} remains open, we found evidence to answer Bukh's question in the affirmative when the chromatic number is replaced by the fractional chromatic number. In fact, we show that this holds not only in expectation but holds for subgraphs of $G$ with high probability, see Theorem \ref{thm:main} below.

Let us first recall the definition of the fractional chromatic number. Let $\mathcal{I}(G)$ be the family of all independent sets of $G$. For each vertex $v\in V(G)$, let $\mathcal{I}(G,v)$ be the family of all those independent sets which contain $v$. For each independent set $I$, consider a nonnegative real variable $y_I$. The \emph{fractional chromatic number} of $G$, denoted by $\chi_f(G)$, is the minimum value of
\begin{equation}
   \sum_{I\in\mathcal{I}(G)} y_I, \quad \hbox{ subject to } \quad \sum_{I\in\mathcal{I}(G,v)} y_I \ge 1 \quad \hbox{ for each } v\in V(G).
\label{eq:chif_LP}
\end{equation}

In this paper, we prove in the affirmative a strengthening of Bukh's question for the fractional chromatic number.

\begin{thm}
\label{thm:main}
Let $t\ge2$ be a rational number and let\/ $G$ be a graph with $\chi_f(G)=t$.
Then for every $p\in (0,1)$ and every $c>0$ we have:
$$
   Pr\Bigl(\chi_f(G_p)\ge \frac{t}{4\log_{1/p}(et)+4+4c}\Bigr) > \frac{1-2p^c}{1-p^c}.
$$
\end{thm}

By taking $c=\log_{1/p}(et)$, we obtain the following:

\begin{cor}
\label{cor:chif(Gp)}
If $\chi_f(G)=t$ and $p\in (0,1)$, then
$$\chi_f(G_p) \ge \frac{t}{8\log_{1/p}(et)+4}$$
with probability at least\/ $1-\tfrac{1}{2t}$.
\end{cor}

\section{Fractional weight and principal vertex-sets}

The proof of Theorem \ref{thm:main} uses the tools presented in this section.
They are based on two concepts, that of a principal subset of vertices and that of a sparse subset. These two notions were used previously in our fractional versions of the Erd\H{o}s-Neumann Lara conjecture (see \cite{MW}) and the Erd\H{o}s-Hajnal conjecture from \cite{EH} for triangle-free subgraphs (see \cite{MW_trianglefree}) that every graph with large chromatic number contains a triangle-free subgraph whose chromatic number is still large.

We let $V=V(G)$, $n=|V|$, and $t = \chi_f(G)$.
By the linear program duality for the definition (\ref{eq:chif_LP}) of $\chi_f(G)$, there exists a non-negative \emph{weight function} $w:V\to\RR^+$, such that $w(V)=t$, and for any $I\in \mathcal{I}(G)$, $w(I)\le 1$. See \cite{GR} for more details.
Here and in the sequel we write $w(A) = \sum_{v\in A} w(v)$ for any vertex set $A\subseteq V$, and call this value the \emph{weight\/} of $A$.

From now on we fix $w$ and assume that the vertices of $G$ are listed as $\VEC v1n$ in the non-increasing order of their weights, i.e. $w(v_{i+1})\le w(v_i)$ for $i=1,\dots,n-1$. For any subset $X$ of $V$, we also rank the elements in $X$ according to the ordering of $V$, and we denote by $X_k$ the subset of the first $k$ elements in $X$. In particular, $V_k=\{\VEC v1k\}$. We extend this notion to any real number $s\ge1$ by setting
$X_s := X_{\lfloor s\rfloor}$.

For a real number $s\ge 1$, a nonempty subset $X$ of a vertex set $Y$ is said to be \emph{$s$-principal\/} in $Y$ if $X\subseteq Y_{s|X|}$. That is, if $X$ has size $m$, then all elements of $X$ are among the first $\lfloor sm\rfloor$ vertices in $Y$. As a kind of opposite property, we say that a subset $X$ of $Y$ is \emph{$s$-sparse} in $Y$ if $X$ contains no $s$-principal subset in $Y$. Note that every subset of an $s$-sparse set in $Y$ is also $s$-sparse in $Y$. When the hosting set $Y$ for $s$-principal or $s$-sparse is not specified, by default it is the whole vertex-set~$V$.

The following condition gives another description of sparse sets that is easier to deal with computationally.

\begin{lem}
\label{lem:sparse characterization}
$X$ is an $s$-sparse set in $Y$ if and only if $|Y_k\cap X|< k/s$ for every integer $k = 1,2,\dots |Y|$. In particular, if $X$ is $s$-sparse in $Y$, then $|X|<|Y|/s$.
\end{lem}

\begin{proof}
It is clear that $X$ is $s$-sparse if and only if for each $r=1,\dots,|X|$, $|Y_{sr}\cap X| < r$.
If $|Y_k\cap X|< k/s$ for every integer $k$, then this holds also for $k=\lfloor sr\rfloor$, implying that
$|Y_{sr}\cap X| < \lfloor sr\rfloor / s \le r$. Conversely, if $|Y_{sr}\cap X| < r$, then $|Y_{sr}\cap X| \le r-1$. Let $s(r-1) < k \le sr$. Then $|Y_k\cap X| \le |Y_{sr}\cap X| \le r-1 < k/s$.
\end{proof}

The next claim about the total weight of an $s$-sparse set will be essential for us.

\begin{lem}\label{lem:sparseislight}
Let $s\ge1$ be a real number.
If $X$ is an $s$-sparse subset of\/ $Y$, then $w(X)\le \frac{1}{s}\,w(Y)$.
\end{lem}

\begin{proof}
Let $\VEC y1r$ be the non-decreasing order of the elements of $Y$ with $r=|Y|$, and let $\VEC x1m$ be the ordering of $X$ with $m=|X|$. Since $X$ is an $s$-sparse subset of $Y$, we have $x_i\not\in Y_{si}$. Hence for $1\le i\le m$, $w(x_i)\le w(y_j)$ if $1\le j\le \lfloor si\rfloor$. Moreover, since $x_i\in Y\setminus Y_{si}$, we also have $w(x_i)\le w(y_j)$ for $j=\lceil si\rceil$.

For a real parameter $z\in(0,|Y|]$, define $f(z)=y_{\lceil z\rceil}$. Then $f(z)\ge w(x_1)$ for $0<z\le s$, $f(z)\ge w(x_2)$ for $s<z\le 2s$, \dots, $f(z)\ge w(x_m)$ for $(m-1)s < z \le ms$.
Therefore,
$$
    s\,w(X) = s\sum_{i=1}^m w(x_i) \le \int_0^{sm} f(z)dz \le \sum_{j=1}^{\lceil sm\rceil} w(y_j) \le w(Y),
$$
which gives what we were aiming to prove.
\end{proof}

\section{Fractional Chromatic number of a random subgraph}

\begin{lem}
\label{lem:5}
Let $p\in (0,1)$, $c>0$, and $s\ge1$ be real numbers.
With probability at least $\frac{1-2p^{c}}{1-p^{c}}$ no $s$-principal set in $V(G)$ with average degree at least $2\log_{1/p}(es)+2c$ in $G$ is independent in $G_p$.
\end{lem}

\begin{proof}
Let $A$ be an $s$-principal vertex-set of cardinality $k$ and with average degree at least $2\log_{1/p}(es)+2c$.  Note that $A$ is independent in $G_p$ if and only if none of the edges of $G(A)$ is present in $G_p$. In other words,
$$
   Pr(A\mbox{ is independent}) = p^{e(G(A))} \le p^{(\log_{1/p}(es)+c)k}=
   \Bigl(\frac{p^c}{es}\Bigr)^k.
$$
Also, as an $s$-principal set, $A$ is a $k$-set contained in $V_{sk}$. Thus, there are at most ${sk\choose k}$ $s$-principal sets with $k$ elements.
Let $B_k^s$ be the event that some $s$-principal $k$-set with average degree at least $2\log_{1/p}(es)+2c$ is independent.
By the above, the probability of $B_k^s$ is at most
\begin{equation}\label{eq:B_s}
   \biggl(\frac{p^c}{es}\biggr)^k{sk\choose k}\le \biggl(\frac{p^c}{es}\biggr)^k(es)^k = p^{ck}.
\end{equation}
Using (\ref{eq:B_s}) we see that the probability that some $s$-principal set with average degree at least $2\log_{1/p}(es)+2c$ is independent in $G_p$ is at most
$$
  \sum_{k=1}^n Pr(B_k^s) \le \sum_{k=1}^n p^{ck} < \frac{p^c}{1-p^c}.
$$
This implies that with probability more than $\frac{1-2p^c}{1-p^c}$ no $s$-principal set of $G$ with average degree at least $2\log_{1/p}(es)+2c$ is independent in $G_p$.
\end{proof}

\begin{lem}
\label{lem:6}
Suppose that $A\subseteq V(G)$ is a vertex-set that contains no $s$-principal sets whose average degree in $G$ is at least $x$.
If every independent subset of $A$ has weight at most $1$, then $A$ has weight at most $2\lfloor x+1\rfloor + \frac{2w(V)}{s}$.
\end{lem}

\begin{proof}
For any vertex-set $X$, we will denote by $e(X)$ the number of edges in the induced subgraph $G(X)$.
For a vertex $v \in A$, let $d^>(v)$ be the number of neighbors of $v$ in $G(A)$ that appear before $v$ in the ordering $\VEC v1n$. For each $i$, we have $e(A_i)=\sum_{v\in A_i} d^>(v)$. Let $L=\{v\in A: d^>(v)\ge x\}$, and assume $L=\{v_{i_1},\dots,v_{i_l}\}$, where $l=|L|$. Then $e(V_{i_j}\cap A)\ge x\cdot j $ for $1\le j\le l$. In particular, if $|V_{i_j}\cap A|\le 2j$, then the average degree $\bar d(V_{i_j}\cap A)$ will be at least $x$. Since $A$ contains no $s$-principal sets with average degree at least $x$, the set $V_{i_j}\cap A$ cannot be $s$-principal in this case. Thus, we have one of the following:

\medskip

(1) $|V_{i_j}\cap A|> 2j$, or

(2) $|V_{i_j}|> s\,|V_{i_j}\cap A|$.

\medskip
\noindent
Let $L_1=\{v_{i_j}\in L:|V_{i_j}\cap A|> 2j\}$, and let $L_2=\{v_{i_j}\in L: |V_{i_j}|\ge t|V_{i_j}\cap A|\}$. Then $L=L_1\cup L_2$.
If $v_{i_j}\in L_1$, then $v_{i_j}$ is not among the first $2j$ elements of $A$. As the $j$-th element in $L_1$ does not appear before $v_{i_j}$, the $j$th element of $L_1$ is not among the first $2j$ elements of $A$. Therefore $L_1$ is a 2-sparse subset of $A$. By Lemma~\ref{lem:sparseislight}, $w(L_1)\le \frac{1}{2}w(A)$.

For each $v_{i_j}\in L_2$, we have $|V_{i_j}\cap A|<\frac{1}{t}|V_{i_j}|$. As $|V_{i_j}\cap L_2|\le |V_{i_j}\cap A|<\frac{1}{t}|V_{i_j}|$,
we see by Lemma \ref{lem:sparse characterization} that $L_2$ is an $s$-sparse subset of $V$. By Lemma~\ref{lem:sparseislight}, $w(L_2)\le \frac{w(V)}{s}$.

Let $S=A-L=\{v\in A: d^>(v)<x\}$. Then we have $w(S)\ge w(A)-w(L_1)-w(L_2)\ge \frac{w(A)}{2}-\frac{w(V)}{s}$. Also, $G(S)$ is an $\lfloor x\rfloor$-degenerate graph, hence $S$ is $\lfloor x+1\rfloor$-colorable. There is at least one independent set $I\subseteq S$ with weight
$$w(I) \ge \frac{w(S)}{\lfloor x+1\rfloor}\ge \frac{w(A)/2-w(V)/s}{\lfloor x+1\rfloor}.$$
As every independent set contained in $A$ has weight at most $1$, we have 
$w(A)\le 2\lfloor x+1\rfloor + 2w(V)/s$.
\end{proof}

\begin{proof}[Proof of Theorem \ref{thm:main}]
We are going to use Lemmas \ref{lem:5} and \ref{lem:6} with $s=t=\chi_f(G)$, and $x=2\log_{1/p}(et)+2c$.
With probability at least $(1-2p^c)/(1-p^c)$, no principal set with average degree in $G$ at least $x$ is independent in $G_p$ by the first lemma. Thus it is sufficient to see that every such subgraph $G_p$ has fractional chromatic number at least $\tfrac{t}{2x+4}$. To see this, we will use the second lemma.

Note that the weight function $w$ defines $\chi_f(G)$, i.e., $w(V)=t$. Define the weight function $w' = w/(2x+4)$ and consider any independent vertex-set $A$ in $G_p$. By Lemma \ref{lem:6},
$$w'(A) = \frac{w(A)}{2x+4} \le \frac{2\lfloor x+1\rfloor + 2w(V)/t}{2x+4} \le 1.$$
This weight function thus justifies that $\chi_f(G_p) \le w'(V) = t/(2x+4)$.
\end{proof}

\end{document}